\documentclass[11pt]{amsart}
\usepackage{fullpage}
\usepackage[centertags]{amsmath}
\usepackage{amsfonts}
\usepackage{amssymb}
\usepackage{amsthm}

\usepackage{pstricks}
\usepackage{epsfig}
\usepackage{pst-grad} 
\usepackage{pst-plot} 

\begin{document}

\newtheorem{teo}{Theorem}[section]
\newtheorem{coro}[teo]{Corollary}
\newtheorem{lema}[teo]{Lemma}
\newtheorem{fact}[teo]{Fact}
\newtheorem{ejem}[teo]{Example}
\newtheorem{obser}[teo]{Observation}
\newtheorem{rem}[teo]{Remark}
\newtheorem{ejer}[teo]{Exercise}
\newtheorem{propo}[teo]{Proposition}
\newtheorem{defi}[teo]{Definition}
\newtheorem{question}{Question}

\def\eqdef{\stackrel{\rm def}{=}}
\def\Proof{{\noindent {\em Proof:\ }}}
\def\fp{\hfill $\Box$}

\def\alex{\mbox{\sf AT}}
\def\baire{{\nat}^{\nat}}
\def\bairez{{\bf Z}^{\omega}}
\def\base{{\mathcal B} }
\def\binary{2^{< \omega}}
\def\cantor{2^{\nat}}
\def\cantorx{2^X}
\def\ca{{\mathcal A}}
\def\cc{{\mathcal C}}
\def\ci{{\,mathcal I}}
\def\cj{{\mathcal J}}
\def\ck{{\mathcal K}}
\def\cL{{\mathcal{L}}}
\def\coA{{\mathbb{A}}}
\def\coB{{\mathbb{B}}}
\def\coL{{\mathbb{L}}}
\def\coK{{\mathbb{K}}}
\def\cl#1{\overline{#1}}
\def\cofin{\mbox{\sf CoFIN}}
\def\ds{{\frak D}}
\def\esptop{(X,\top)}
\def\fsig{F_{\sigma}}
\def\fsd{F_{\sigma\delta}}
\def\fin{\mbox{\sf Fin}}
\def\filter{{\mathcal F}}
\def\afilter{\vec{\filter}}
\def\gfilter{\mathcal G}
\def\ged{G_\delta}
\def\hilbert{[0,1]^{\nat}}
\def\hm {O\check{c}(\mbox{$\tau$-closed, $\tau$})}
\def\ideal{{\mathcal I}}
\def\inte#1{\buildrel {\;\circ} \over #1}
\def\intseq{{\bf Z}^{<\omega}}
\def\ikl{[K,L]}
\def\iko{[K,O]}
\def\leqtau{\leq_{\tau}}
\def\k#1{{\mathcal K}(#1)}
\def\lx{<_{\scriptscriptstyle X}}
\def\ly{<_{\scriptscriptstyle Y}}
\def\rl{\Re_{l}}
\def\nat{\mathbb{N}}
\def\N{\mathbb{N}}
\def\nwd{\mbox{\sf nwd}}
\def\opcla{{\bf cl}}
\def\nin{\not\in}
\def\oc#1{O\check{c}({\fin},{#1})}
\def\oppoint{$\mbox{\bf op}$}
\def\oqpoint{$\mbox{\bf oq}$}
\def\ptres{\emptyset\times\mbox{FIN}}
\def\prodxa{\prod_\alpha \xa}
\def\power#1{{\mathcal P}(#1)}
\def\powerx{{\mathcal P}(X)}
\def\ppoint{$\mbox{p}^+$}
\def\pp{\mbox{$\mathbf  p^+$}}
\def\Pm{\mbox{$\mathbf  p^-$}}
\def\qpoint{$\mbox{q}^+$}
\def\qq{\mbox{$\mathbf{q}^+$}}
\def\Q{\mathbb{Q}}
\def\qed{\hfill \mbox{$\Box$}}
\def\R{\mathbb{R}}
\def\rpoint{$\mbox{rp}^+$}
\def\sd{\scriptstyle{\triangle}}
\def\seq{{\omega}^{<\omega}}
\def\sfan{S(\omega)}
\def\somega{S_\omega}
\def\soi{S(\ideal)}
\def\su{\subseteq}
\def\swc{\overline{S_{\Omega}}}
\def\Sq{\mbox{$\mathbf{p}^-$}}
\def\Rq{\mbox{$\mathbf{ws}$}}
\def\sw{S_{\Omega}}
\def\sx{\scriptscriptstyle X}
\def\sy{\scriptscriptstyle Y}
\def\tauf{\tau_{\scriptscriptstyle \filter}}
\def\tauaf{\tau_{\scriptscriptstyle \afilter}}
\def\tauc{\tau_{\scriptscriptstyle \cc}}
\def\tauv{\tau_{V}}
\def\ti{T(\ideal)}
\def\topy{\top_{\sy}}
\def\topo{topolog\II a }
\def\top{\tau}
\def\ult{\mathcal U}
\def\vkl{{\mathcal V}_{\scriptscriptstyle{K,L}}}
\def\veba#1#2{{\mathcal V}_{\scriptscriptstyle{#1,#2}}}
\def\iveba#1#2{[#1,#2]}
\def\ve{{\mathcal V}}
\def\ufilter{\mathcal U}
\def\w1{\omega_1}
\def\wpp{\mbox{$\mathbf{wp}^+$}}
\def\xa{X_\alpha}
\def\Z{\mathbb{Z}}

\def\sequen{\nat^{<\bf \scriptscriptstyle \omega}}
\def\casig{=^*}

\def\pie{\Pi_1^1}
\def\pieb{\boldsymbol{ \Pi}_1^1}
\def\pie03{{\mathbf{\Pi}}_3^0}
\def\sigb{{\bf \Sigma}_1^1}
\def\delb{\boldsymbol{ \Delta}_1^1}
\def\si12{\Sigma^{1}_{2}}
\def\del{\Delta_1^1}
\def\zig{\Sigma_1^1}
\def\pai2{\mathbf{\Pi}_2^0}
\def\sig2{{\bf \Sigma}^{0}_{2}}

\title{On the Ellis semigroup of a cascade on  a compact  metric countable space}
\author{Andres Quintero   \and Carlos Uzc\'ategui}
\address{Escuela de Matem\'aticas, Facultad de Ciencias, Universidad Industrial de
	Santander, Ciudad Universitaria, Carrera 27 Calle 9, Bucaramanga,
	Santander, A.A. 678, COLOMBIA.}
\email{quintero.andres7@gmail.com}

\address{Escuela de Matem\'aticas, Facultad de Ciencias, Universidad Industrial de
	Santander, Ciudad Universitaria, Carrera 27 Calle 9, Bucaramanga,
	Santander, A.A. 678, COLOMBIA.}
\email{cuzcatea@saber.uis.edu.co}

\date{\today}

\thanks{The authors thank La Vicerrector\'ia de Investigaci\'on y Extensi\'on de la Universidad Industrial de Santander for the financial support for this work,  which is part  of the VIE project  \# 2422}



\begin{abstract} Let $X$ be a compact  metric countable space, let  $f:X\to X$ be a homeomorphism and let $E(X,f)$ be its Ellis semigroup. Among other results we show that the following statements are equivalent: (i) $(X,f)$   is equicontinuous, (ii)  $(X,f)$ is distal and (iii) every point is periodic.  We use this result to give a direct  proof of a theorem of Ellis saying that $(X,f)$ is distal if, and only if,  $E(X,f)$ is a group.  
\end{abstract}

\maketitle

\section{Introduction}
Let $X$ be a compact metric space and $f:X\to X$ a homeomorphism. The dynamical system $(X,f)$ is usually called a cascade.  Ellis introduced \cite{Ellis1960} the enveloping semigroup (also called the Ellis semigroup) $E(X,f)$  of the  dynamical system $(X,f)$ as the closure of $\{f^n:\; n\in \Z\}$ inside $X^X$ with the product topology.  $E(X,f)$ is a compact semigroup with composition as the algebraic operation. There is an extensive literature about the enveloping semigroup (see, for instance,  \cite{Auslander1988,Ellis2013,Glasner2007}). A homeomorphism $f:X\to X$ is  {\em distal} if 
\[
\inf \{d(f^n(x), f^n(y)):\; n\in \Z\}>0
\]
for every pair of different points $x$ and $y$ of $X$.    
A well known characterization of distality is as follows (see \cite{Auslander1988} page 69). The system  $(X,f)$ is distal if, and only if, every point of $(X\times X, f)$ is almost periodic and also if, and only if, $E(X,f)$ is a group. Distality is a weakening of another very important notion. 
A system $(X,f)$ is {\em equicontinuous}, if $\{f^n:\; n\in \Z\}$ is an equicontinuous family. It  is easy to verify that every equicontinuous system is distal. The converse is not true in general. When  $(X,f)$ is equicontinuous, every function in $E(X,f)$ is obviously  continuous and such systems are called {\em WAP}.  It is also well known that 
a system  $(X,f)$ is equicontinuous if, and only if, $(X,f)$ is distal and WAP (see \cite{Auslander1988} page 69). 

Continuing the work initiated in \cite{GRU2015,GRU2016},   we are interested on  dynamical systems  on  compact metric countable spaces.  Our main result is that, for  a compact metric countable space $X$ and a homeomorphism $f:X\to X$, the following statements are equivalent:

\begin{enumerate}
\item $(X,f)$ is equicontinuous.

\item $(X,f)$ is a distal.

\item Every point of $X$ is periodic.

\item There is $g\in E(X,f)\setminus \{f^n:\; n\in \Z\}$ which is  a homeomorphism.

\item For all $\varepsilon>0$ there is $l$ such that  $d(x,f^{nl}(x))<\varepsilon$ for all $x\in X$ and all $n\in \nat$ (where $d$ is the metric on $X$).

\item $E(X,f)=\overline{\{f^n:\; n\in\nat\}}$.
\end{enumerate}

Even though our main interest is on countable spaces, when possible, we present the results for arbitrary compact metric spaces.  The following implications hold in general: $(5)\Rightarrow (1)$ (see the proof of Theorem \ref{uniforcasiperiodico}),  $(2)\Rightarrow (6)$ (see Theorem \ref{distal-Iguales-E}), $(3)\Rightarrow (2)$ (see the proof of Theorem \ref{distal}) and $(1)\Rightarrow (4)$ (see the proof of Theorem \ref{uniforcasiperiodico}). On the other hand,  it is well known that in general   $(2)\not\Rightarrow (1)$ (see Example \ref{abelianNonWAP}) and  $(4)\not\Rightarrow (1)$.  In addition,  $(1)\not\Rightarrow (3)$ (for instance a rotation of a circle with an irrational angle). The  fact that  $(4)$ and $(1)$ are equivalent reminds a similar  phenomena observed in dynamical systems on the interval \cite{Szuca} or more  generally on $k$-ods \cite{VidalGarcia2019}. They showed that if $X$ is a $k$-od and $f:X\to X$ is continuous,  then $(X,f)$ is equicontinuous if, and only if, there is  $g\in \overline{\{f^n:\;n\in\nat\}}\setminus\{f^n:\; n\in\nat\}$ which is continuous. We do  not know if $(3)$ implies $(1)$ holds in general (it has been shown for dynamical systems on dendrites \cite[Theorem 4.14]{JavierMichael2019}).

For distal dynamical systems on countable spaces, we give a neat representation of the inverse of an arbitrary element  of $E(X,f)$  in terms of $p$-iterates of $f$ (where $p$ is an ultrafilter over $\nat$). That representation is also valid for arbitrary metric equicontinuous dynamical systems.  This fact serves to give a different proof of a well known result of Ellis saying that $E(X,f)$ is a group if, and only if, $(X,f)$ is distal.  

\section{Preliminaries}
\label{preliminaries}

$\nat$ denotes the set of non negative integers. In this paper, $X$ will always be a compact metric space. A  dynamical system is a pair $(X,f)$ where  $f:X\to X$ is a continuous map and most of the time $f$ will be a homeomorphism.  The {\it orbit} of $x$, denoted by $\mathcal O_f(x)$, is
the set $\{f^n(x):n\in\mathbb N\}$, where $f^n$ is $f$ composed
with itself $n$ times.  When $f$ is a bijection, we also use $\mathcal O^{\mathbb{Z}}_f(x)$ for the  $\Z$-orbit of a point, i.e. the set $\{f^n(x):n\in\mathbb Z\}$. A point $x\in X$ is called a {\it periodic
point} of $f$ if there exists $n\geq 1$ such that $f^n(x)=x$ and the least such $n$ is called the {\em period} of $x$. We denote by $P_f$ the collection of possible periods, i.e. $n\in P_f$ if there is $x$ with period $n$. 

The main results presented are about compact metric countable spaces. Notice that there is some redundancy here, as any compact countable Hausdorff space is necessarily metrizable.  It is well known that every compact metric countable space  is  scattered (i.e. every non empty subset has an isolated point), this follows from the classical Cantor-Bendixson's theorem (see e.g. \cite[30E]{Willard}).   We recall  the definition of the Cantor-Bendixson rank  which is 
very useful to study scattered spaces.   The set of all accumulation points of $X$,
the derivative of $X$, is denoted by $X'$.  For a successor ordinal $\alpha=\beta +1$,  we let
$X^{(\alpha)}=(X^{(\beta)})'$ and for limit ordinal $\alpha$ we
let
 $X^{(\alpha)}=\bigcap_{\beta<\alpha}X^{(\beta)}$. The
{\it Cantor-Bendixson rank} of $X$ is the first ordinal
$\alpha<\omega_1$  such that $X^{(\alpha+1)}=X^{(\alpha)}$; notice that  when $X$ is countable, then $ X^{(\alpha)}=\emptyset$ and   for $X$ compact, $\alpha$ is a successor ordinal. The {\it
Cantor-Bendixson rank} of a point $x\in X$, denoted by $cbr(x)$, is the
first ordinal $\alpha<\omega_1$  such that $x\in X^{(\alpha)}$ and
$x\notin X^{\alpha+1}$.  The cb-rank is clearly a topological invariant.

The Stone-\v{C}ech compactification $\beta(\mathbb N)$ of  $\mathbb N$
with the discrete topology will be identified with the set of
ultrafilters over $\mathbb N$. Its remainder, ${\mathbb N}^*=
\beta({\mathbb N})\setminus \mathbb{N}$,  is  the set of all non principal
ultrafilters on $\mathbb N$, where, as usual, each natural number
$n$ is identified with  the principal ultrafilter consisting of all
subsets of $\mathbb N$ containing $n$. For $A\subseteq \mathbb N$, $A^{\wedge}$  (resp. $A^*$) denotes the collection of all ultrafilter  (resp. non principal ultrafilter)  containing $A$. Given $n\in\N$ and $p\in \beta\nat$, it is a simple observation that there is $0\leq r<n$ such that $p\in (n\N+r)^{\wedge}$, where $n\N=\{nk:\; k\in \N\}$.  This will be used several times in the sequel. 

An important tool  is the notion of the $p$-limit of a sequence respect to an ultrafilter $p$. 
Given $p\in \mathbb{N}^*$ and a sequence
$(x_n)_{n}$ in a space $X$, we say that a point $x\in
X$ is the $p$-{\it limit point} of the sequence, in symbols $x
=p$-$\lim_{n\in\nat}x_n$, if for every neighborhood $V$
of $x$, $\{n\in\mathbb{N}: x_n\in V\} \in p$. Observe that a
point $x\in X$ is an accumulation point of a countable subset
$\{x_n:\,n\in\mathbb{N}\}$ of $X$ if, and only if, there is $p\in \mathbb{N}^*$
such that $x = p$-$\lim_{n\in\nat}x_n$. It is not hard to
prove that each sequence in a compact space always has a $p$-limit
point for every $p\in \mathbb{N}^*$.

The notion of a $p$-limit has been used in different context. Hindman  \cite{Hindman} presents some of the history of this notion and its applications to Ramsey theory  and Blass  \cite{Blass} gives results on the connection of $p$-limits and topological dynamic.  They  formally established
the connection between  ``the iteration in to\-po\-lo\-gi\-cal
dynamics'' and ``the convergence with respect to an ultrafilter''
by considering a more general iteration of the function $f:X\to X$ as
follows. For $p\in\N^*$, the
$p$-iterate of $f$ is the function $f^p: X\rightarrow X$ defined
by 
\[
f^p(x) = \mbox{$p$-$\lim_{n\in\nat} f^n(x)$},
\]
for all $x\in X$. In \cite{GarciaSanchis2007,GarciaSanchis2013} the reader can find several interesting properties of $p$-limits and $p$-iterates.

Our main interest are systems $(X,f)$ where $f$ is a homeomorphism. Nevertheless, we also will  consider the situation when $f$ is just a continuous map.  In this case, the Ellis semigroup  associated to $(X,f)$ is defined as follows:
\[
E(X,f, \N)=\overline{\{f^n:\; n\in \N\}}
\]
where the closure is taken in $X^X$ with the product topology. When $f$ is a homeomorphism, the Ellis semigroup of $(X,f)$ is defined by
\[
E(X,f, \Z)=\overline{\{f^n:\; n\in \Z\}}.
\]
In both cases, the Ellis semigroup is a compact left topological semigroup with composition as the algebraic operation (i.e. for a given $g$ in the semigroup, the function $h\mapsto h\circ g$ is continuous).  
It is evident  that for a homeomorphism $f$ the following holds:
\[
E(X,f, \Z)= E(X,f, \N) \cup E(X, f^{-1}, \N).
\]

A key tool for our work is a representation of $E(X,f, \N)$ in terms of ultrafilters on $\N$ is  (for a proof see, for instance, \cite{GarciaSanchis2013}):
\[
E(X,f, \N)=\{f^p:\; p\in \beta\N \},
\]
\[
f^p\circ f^q=f^{q+p}\text{ for $p,q\in \beta\nat$}.
\]
where the sum of ultrafilter is defined by 
\[
A\in p+q \;\;\text{if}\;\; \{n\in\nat: \{m\in \nat:\; n+m\in A\}\in p\}\in q.
\]
It is elementary to show that $\overline{\mathcal{O}^\Z_f(x)}$ is equal to $\{h(x):\; h\in E(X,f,\Z)\}$. Equivalently,  $y\in \overline{\mathcal{O}^\Z_f(x)}$ if, and only if, $y=f^p(x)$ or $y={(f^{-1})}^p(x)$ for some $p\in \beta\nat$.

It is easy to verify that $(X,f)$ is distal if, and only if,  every function in $E(X,f,\Z)$ is injective. We will use this equivalence  through out the paper.  When $f$ is just a continuous function, we  call the system $(X,f)$  {\em positively distal}, if every function in $E(X,f,\N)$ is injective.  We will show later that, for $X$ countable and  $f$  a homeomorphism,  there is no conflict with those two versions: Every function in  $E(X,f,\N)$ is injective,  if, and only if, every function in $E(X,f,\Z)$ is injective.  

A system $(X,f)$ is {\em WAP} (weakly almost periodic), if every function in $E(X,f,\Z)$  is continuous.   We say that $(X,f)$ is {\em equicontinuous}, if for every $\varepsilon>0$, there is $\delta>0$ such that for all $x, y \in X$ with $d(x,y)<\delta$ we have that $d(f^n(x),f^n(y))<\varepsilon$ for all $n\in \Z$.   It is easy to verify  that an equicontinuous system is also WAP, the converse is not true. As it was already mentioned,  for an arbitrary compact metric spaces, any distal WAP system is  equicontinuous (see \cite{Auslander1988} page 69). 
When  $f$ is not a bijection, the notion of a WAP or equicontinuous system  is defined in terms of $E(X,f,\nat)$. 

We end this section by presenting some technical lemmas about the $p$-iterates that we will use in the sequel. 

\begin{lema}\cite[Lemma 2.2]{GRU2015}
\label{GRU1} Let $X$ be a compact metric space, $f:X\to X$ be a continuous map and $x\in X$.
Suppose $x$ is periodic with period $n$ and $0\leq l<n$. If  $p\in  (n\nat +l)^*$, then  
$f^p(x) = f^l(x)$.
\end{lema}

It is known that $E(X,f,\Z)$ is abelian when it is WAP (see \cite{Auslander1988} page 55). 
We include a sketch  of the proof in terms of $p$-iterates for the sake of completeness. 

\begin{lema}
\label{p-lim}
Let $X$ be a compact metric space, $h,g:X\to X$ be continuous maps, $(x_n)_n$ be a sequence in $X$ and $p\in \beta\N$. 
\begin{itemize}
\item[(i)]  $h(p$-$\lim_n x_n)=p$-$\lim_n h(x_n)$.
\item[(ii)] If $h\circ g=g\circ h$, then  $h \circ g^p=g^p\circ h$.
\item[(iii)] If $f^p$ is continuous and $f\circ g=g\circ f$, then $f^p\circ g^q=g^q\circ f^p$ for every $q\in\beta\N$.

\item[(iv)] If $(X,f)$ is WAP, then $E(X,f,\nat)$ is abelian. The same holds for $E(X,f,\Z)$ when $f$ is a homeomorphism.
\end{itemize}
\end{lema}

\proof (i) It is elementary and left to the reader.  (ii) Let $x\in X$, we have
$$
h(g^p(x))=h(p\mbox{-$\lim_m g^m(x))=p$-$\lim_m h(g^m(x))=p$-$\lim_n g^m(h(x))=g^p(h(x))$}.
$$
(iii) Since $g$ is continuous, from (ii) we have that $ g\circ f^p =f^p\circ g$. Since $f^p$ is continuous,  by (ii) 
$f^p\circ g^q=g^q\circ f^p$ for every $q\in \beta\N$. (iv) follows from (iii).
\endproof

We present below an example  of a non abelian Ellis semigroup. 

\begin{ejem}\label{ejem1}
{\em Let $X=\mathbb{Z}\cup\{-\infty, +\infty\}$ with the topology where $n\to +\infty$ and $-n\to -\infty$, i.e, $X$ is the union of two convergent sequences.  Consider the function given by $f(\pm \infty)= \pm \infty$ and  $f(n)=n+1$ for all $n\in \mathbb{Z}$.  Let $f^+\in X^X$ be the function defined by $f^+(x)=+\infty$ for all $x\neq -\infty$ and  $f^+(-\infty)=-\infty$. Analogously we define $f^-\in X^X$ interchanging $+\infty$  and  $-\infty$. 
Then $f^p= f^{+}$ and $(f^{-1})^p= f^{-}$ for every non principal ultrafilter $p$ on $\nat$.   
Hence
\[
E(X,f,\Z)=\{f^n:\;n\in \N\}\cup \{f^+,f^-\}.
\]
We have that $f^{+}\circ f^{-}\neq f^{-}\circ f^{+}$.
}

\end{ejem}

In Example \ref{abelianNonWAP} we provide a non WAP system with an abelian enveloping semigroup.

\section{Distal dynamical systems}
Now we are going to show a key result to understand equicontinuity on countable spaces. 

\begin{teo}
\label{distal}
Let $X$ be a compact metric countable space. 
Let $f:X\to X$ be a homeomorphism. The following are equivalent. 
\begin{itemize}
\item[(i)] $(X,f)$ is  positively distal.
\item[(ii)] Every point of $X$ is periodic.  
\end{itemize}
\end{teo}

To put this result  in a proper context, we need to recall some fundamental notions about topological dynamics. Let $(X,f)$ be a dynamical system with $X$ compact metric and $f:X\to X$ a homeomorphism. 
A subset $A\su X$ is {\em minimal}, if $A$ is closed non empty, $f[A]\su A$ and  whenever $B\su A$ is closed, non empty and $f[B]\su B$, then $A=B$. By  the compactness of $X$ and by the usual Zorn's Lemma argument, there are minimal  sets (see \cite{Auslander1988}). 
A subset $A$ of $\mathbb{Z}$ is said to be  {\em syndetic}, if there is a finite set $K\subset \mathbb{Z}$ such that
$\mathbb{Z}=A+K=\{a+k: a\in A, k\in K\}$. We say that $x\in X$  is an {\em almost periodic} point if for every neighborhood $U$ of $x$, there is a syndetic set $A$ such that $f^n(x)\in  U$ for all $n\in A$.  A point is almost periodic if, and only if, the closure of its orbit is a minimal set (see \cite{Auslander1988} page 11). Every  point of a distal system  $(X,f)$ has to be almost periodic (see \cite{Auslander1988} page 68).   A characterization of distality is as follows:  $(X,f)$ is distal if, and only if, every point of $(X\times X, f)$ is almost periodic where the product flow is defined coordinate wise  (see  \cite{Auslander1988} page 69). It follows  that every point of $X$ is periodic if, and only if, $(X,f)$ is distal and every minimal set is finite. On the other hand, equicontinuity corresponds to a stronger form of almost periodicity. A system $(X, f)$ is {\em uniformly almost periodic}, if for every $\varepsilon>0$ there is a syndetic set $A$ such that $d(x,f^n(x))<\varepsilon$ for all $x\in X$ and $n\in A$. Then $(X,f)$ is equicontinuous if, and only if, it is uniformly almost periodic (see \cite{Auslander1988} page 36).  We will see later that for equicontinuous systems on countable spaces the syndetic set $A$ is equal to the collection of multiples of a natural number (see Theorem \ref{uniforcasiperiodico}). 

We remark that Theorem \ref{distal} only holds for  $X$  countable. We start by showing that in countable spaces the only minimal sets are the periodic orbits. 

\begin{lema}
\label{minimal}
Let $X$ be a compact metric countable space. 
If $f:X\to X$ is a continuous map and $A\su X$ is minimal, then $A$ is the orbit of a periodic point. 
\end{lema}

\proof Let $A$ be  a minimal set and suppose there is $x\in A$ with an infinite orbit. Then $A=\cl{{\mathcal O}_f(x)}$.  Since $X$ is scattered, there is $y\in A$ isolated in $A$ and necessarily $y\in {\mathcal O}_f(x)$. Let $B=\cl{{\mathcal O}_f(f(y))}$, then $B\su A$ is a nonempty closed invariant  set, $y\not\in B$  and thus $B\neq A$, a contradiction.  
\endproof

Let $(X,f)$ be a countable dynamical system with $f$ a homeomorphism. The considerations above imply  that any almost periodic point of $(X\times X, f)$ is periodic and therefore, for a distal system $(X,f)$,  every point of $X$ is periodic. This  provides a proof of Theorem \ref{distal}, nevertheless, we present a direct argument  that  uses neither minimal sets nor almost periodic points. 

\begin{lema}
\label{cbrank}  Let $X$ be a compact metric countable space. 
The collection $M$ of points in $X$ with maximal cb-rank is finite and non empty. Moreover, if $f:X\to X$ is a homeomorphism, then $f(M)=M$ and every point in $M$ is periodic.  
\end{lema}

\proof Since $X$ is a countable compact metric space, its Cantor-Bendixson rank is a successor ordinal, that is to say, there is an ordinal  $\alpha$ such that  $X^{(\alpha)}\neq \emptyset$ and $X^{(\alpha+1)}=\emptyset$. So $M=X^{(\alpha)}$. By compactness $M$ is finite, otherwise, $M$ would have  an accumulation point which would belong to $X^{(\alpha+1)}$. Suppose now that $f:X\to X$  is a homeomorphism.  Since the cb-rank of a point is preserved under homeomorphism,  $f[M]= M$. Since $M$ is  a finite set,  every point in $M$ is periodic.   
\endproof

\begin{lema}
\label{no-distal} 
Let $X$ be a compact metric countable space and  $f:X\to X$ be a homeomorphism.  If there is  a point with dense orbit, then  $(X,f)$ is not positively distal.
\end{lema}

\proof Let $M$  be the set of points in $X$ with maximal cb-rank. By Lemma \ref{cbrank}, $M$ is a non empty finite set  and all its points  are periodic. Suppose the orbit of $x$ is dense. Then $x$ is isolated and has cb-rank equal to $0$. We will show that there is $y\in M$ and $p\in \beta\N$ such that $f^p(y)=f^p(x)$.  Let $d\in M$, then there is $p\in \beta\N$ such that $d=f^p(x)$. Let $n$ be the  period of $d$. Let $0\leq r<n$ be such that $p\in (n\N+r)^*$. Notice that  $f^{n-r}(d)$ has period $n$, thus  by Lemma \ref{GRU1}  we have $f^p(f^{n-r}(d))=f^r(f^{n-r}(d))=f^n(d)=d$. Let $y=f^{n-r}(d)$. Since $y\in M$,  $y$ is not isolated and thus $y\neq x$. Hence $f^p$ is not injective and thus $(X,f)$ is not positively distal. 

\endproof

\noindent {\em Proof of Theorem \ref{distal}:}  Suppose (i) holds and (ii) does not. Let $x$ be a non periodic point. Since $f$ is injective, the orbit of $x$ is infinite. Let $Y=\overline{\mathcal{O}_f(x)}$ and $g:Y\to Y$ be the restriction of $f$ to $Y$.  Since $(X,f)$ is positively distal, $(Y,g)$ is also positively distal (just observe that $\inf \{d(g^n(x), g^n(y)):\; n\in \Z\}>0$ for all $x,y\in Y$). But $x$ has a dense orbit in $Y$ which contradicts Lemma \ref{no-distal}. 

The converse implication is known, but we present a proof for the sake of completeness. Suppose every point is periodic. Let $x,y\in X$ and $p\in\beta\N$ be such that $f^p(x)=f^p(y)$. Since the orbit of $x$ and $y$ are finite, necessarily $\mathcal{O}_f(x)=\mathcal{O}_f(y)$.  Let $n$ be the period of $x$ and $0\leq r<n$ be such that $y=f^r(x)$. By Lemma \ref{p-lim}(ii) (with $h=f^r$ and $g=f$), $f^p(y)=f^p(f^r(x))=f^r(f^p(x))$. Since we assumed $f^p(x)=f^p(y)$,  we have $f^r(f^p(x))=f^p(x)$ and therefore $r=0$, as  $f^p(x)$ has period $n$.  Hence $x=y$. 

\qed

If $f$ is a homeomorphism and every point is periodic, then by Theorem \ref{distal},  $(X,f^{-1})$ is positively distal. So we have the following 

\begin{coro}
\label{distal2}
Let $X$ be a compact metric countable space. 
Let $f:X\to X$ be a homeomorphism.  $(X,f)$ is positively distal if, and only if, it is distal. 
\end{coro}

It is natural to wonder when  $E(X,f,\nat)$ and $E(X,f,\Z)$ are equal, we show next that this happens for any distal system. We need an auxiliary result. 

\begin{lema}
\label{p-periodico}
Let $X$ be a compact metric space and $f:X\to X$ be a homeomorphism. 
The following are equivalent.
\begin{enumerate}
 \item[(i)] $E(X,f,\nat)= E(X,f,\Z)$.
 \item[(ii)] There is $p\in \beta\nat\setminus\{0\}$ such that $f^p=1_X$.
\end{enumerate}
\end{lema}

\begin{proof}
Suppose (i) holds. Then $f^{-1}\in E(X,f,\nat)$, thus there is $q\in \beta\nat$ such that $f^q=f^{-1}$. Then $f^{q+1}=1_X$.
Conversely, suppose (ii) holds and let $p\in \beta\nat\setminus\{0\}$ be such that $f^p=1_X$. It suffices to show that $f^{-1}\in E(X,f,\nat)$. If $p$ is principal, then there is $n\geq 1$ such that $f^n=1_X$. Clearly we can assume $n\geq 2$. Then   $f^{-1}=f^{n-1}\in E(X,f,\nat)$. Now suppose that $p$ is not principal. Let $q$  be $\{A\su\nat:\; A+1\in p\}$. Then $q$ is an ultrafilter and $1_X = f^p=f\circ f^q=f^q\circ f$. Thus $f^{-1}=f^q$ and we are done.
\end{proof}

\begin{teo}
\label{distal-Iguales-E}
Let $X$ be a compact metric space and  $f:X\to X$ be a distal homeomorphism.  Then    $E(X,f,\nat)=E(X,f,\Z)$.
\end{teo}

\begin{proof}
It is well known that there are idempotent ultrafilters, i.e., ultrafilters $p$  such that  $p+p=p$ (see for instance \cite{Blass}). Let $p$ be an idempotent ultrafilter, then $f^p\circ f^p=f^{p+p}=f^p$. Since $f$ is distal, $E(X,f, \Z)$ is a group (see \cite{Auslander1988} page  69) and  thus $f^p=1_X$. Clearly such $p$ is non principal, then the conclusion follows from Lemma \ref{p-periodico}.
\end{proof}

The proof of the previous result used that the enveloping semigroup is in fact a group when the system is distal. In some cases, as for $X$ countable,  we do not need to use this fact. We include the following lemma just to have a self contained presentation of  our results when  $X$ is a compact metric countable space. 

\begin{lema}
\label{igualesE}
Let $X$ be a compact metric space and  $f:X\to X$ be a homeomorphism. If every point is periodic, then   there is $p\in\nat^*$ such that $f^p=1_X$.
\end{lema}

\proof  It is clear that  $\{(k\nat)^*:\; k\in\nat\}$ has the finite intersection property, thus there is an ultrafilter $p$ such that 
$$
p\in \bigcap_{k\in P_f} (k\nat)^*.
$$ 
Since every point is periodic, for any such $p$ we have that  $f^p(x)=x$ for all $x$ (by Lemma \ref{GRU1}).
\endproof

We show now that on countable spaces, Theorem \ref{distal-Iguales-E} is an equivalence.  However, we do not know if this holds in general.  We need an auxiliary result which is interesting in itself.

\begin{lema}
\label{todosperiodicos}
Let $X$ be a compact metric countable space and $f:X\to X$ be a homeomorphism. Suppose $f^p$ is a homeomorphism for some non principal $p\in \beta\nat$. Then every point of $X$ is periodic. 
\end{lema}

\proof  Suppose there is $x\in X$ non periodic. Clearly all $f^n(x)$, with $n\in \Z$, have the same  cb-rank. Since $p$ is not principal,  $f^p(x)$ has larger  cb-rank than $x$, but this contradicts that $f^p$ is a homeomorphism. 
\endproof

The previous result is no longer valid if $f^p$ is not assumed to be injective  as we show next.

\begin{ejem}
Consider  $X$ to be a convergent sequence, say $X=\{x_n:\;n \in \N\}\cup \{\infty\}$.
Let $f$ be defined as follows: $f(x_{2n+1})=x_{2n-1}$ for all $n\geq 2$, $f(x_1)=x_0$, $f(x_{2n})= x_{2n+2}$ for all $n\geq 0$ and $f(\infty)=\infty$.  Then $E(X,f,\Z)$ is equal to $\{f^n:\; n\in \Z\}\cup \{\overline{\infty}\}$, where $\overline \infty$ is the constant function equal to $\infty$. This is a WAP system and there are no periodic points except $\infty$. 
\end{ejem}

\begin{teo}
\label{ZNdistality}
Let $X$ be a compact metric countable space.
Let $f:X\to X$ be a homeomorphism. The following are equivalent.
\begin{itemize}
\item[(i)]  Every point of $X$ is periodic.
\item[(ii)] $E(X,f,\nat)=E(X,f,\Z)$.
\end{itemize}
\end{teo}

\proof  Suppose (i) holds. Then (ii) follows from Lemmas \ref{p-periodico} and \ref{igualesE}. Suppose that (ii) holds. Then $f^{-1}\in E(X,f,\nat)$. Hence there is a   $p\in \beta\nat$ such that $f^{-1}=f^p$. If $p$ is principal, then there is $k\in \nat $ such that  $f^p=f^k$. Hence $f^{k+1}(x)=x$ for all $x\in X$.   If $p$ is not principal, the conclusion follows by Lemma \ref{todosperiodicos}. 
\endproof

\section{The inverse of an element of $E(X,f,\Z)$.}

Now we are going to represent the inverse of an element of the Ellis semigroup.
The idea is to show that in some systems the following equality hold for every $p\in \beta\nat$.
\begin{eqnarray}
\label{p-inversos}
(f^p)^{-1}=(f^{-1})^p.
\end{eqnarray}

\begin{lema}
\label{p-inversos2}
Let $X$ be a compact metric space and $f:X\to X$ be a homeomorphism. Let $x\in X$ be a periodic point and $p\in \beta\N$. Then  
$$
(f^p\circ (f^{-1})^p)(x)= ( (f^{-1})^p\circ f^p)(x)=x.
$$
\end{lema}

\proof If $p$ is a principal ultrafilter, then for some $m$, $f^p=f^m$  and $(f^{-1})^p=f^{-m}$ and  thus the conclusion is obvious. Suppose now that $p\in \nat^*$ and $n$ is the period of $x$. Let $0\leq r<n$ be such that $p\in (n\N+r)^*$. Notice that for any $z\in X$ of period $n$, by Lemma \ref{GRU1}, we have that  $f^p(z)=f^r(z)$ and  $(f^{-1})^p(z)=f^{-r}(z)$.  Since $f^p(x)=f^r(x)$ and $f^r(x)$ also  has period $n$,  $( (f^{-1})^p\circ f^p)(x)=(f^{-1})^p(f^r(x)) =f^{-r}(f^r(x))=x$.  Analogously, we have $(f^p\circ (f^{-1})^p)(x)= x$.
\endproof

\begin{lema}
\label{puntual}
Let $X$ be a compact metric space and $f:X\to X$  be a homeomorphism. Suppose $E(X,f,\Z)$ is WAP and $x$ is a point with dense $\Z$-orbit.  If $f^{n_i}(x)\to f^p(x)$ for an ultrafilter $p$,  then  $(f^{n_i})_i$ converges pointwise to $f^p$. 
\end{lema}

\begin{proof} Let $y \in X$.  We claim that there exists a function $h \in E(X,f,\Z)$ such that $y=h(x)$. In fact, as $x$ has dense $\Z$-orbit,  there are two cases to consider: (i) $y=\lim_k f^{m_k}(x)$ or (ii) $y=\lim_k f^{-m_k}(x)$ for some sequence  $(m_k)_k$ of natural numbers. Both are treated analogously. Suppose (i) holds. Let  $q$ be an ultrafilter containing $\{m_k\}_k$ and take $h=f^q$.  Since the system is WAP, $h$ is continuous. Therefore $h(f^{n_i}(x))\to h(f^p(x))$. By Lemma \ref{p-lim} $E(X,f)$ is abelian, thus $f^{n_i}(h(x))\to f^p(h(x))$ and we are done. 
\end{proof}

\begin{lema}
\label{inverse-orbita-densa}Let $X$ be a compact metric space and $f:X\to X$ be a  homeomorfism.  Suppose $(X,f)$ is an equicontinuous dynamical system and  $x$ is a point with a dense $\Z$-orbit.  Then  $(f^p)^{-1} = (f^{-1})^p$ for every $p \in \beta\nat$. 
\end{lema}

\begin{proof}
Since $\left\{{f^n: n \in \Z}\right\}$ is equicontinuous, the pointwise topology and the uniform topology on $E(X, f, \Z)$ are the same (see for instance \cite{Auslander1988} page 52). So given $p \in \beta\nat$ non-principal, there exists a strictly increasing  sequence of natural numbers $(n_i)_{i \in \N}$  such that $f^{n_i}$ converges uniformly to $f^p$. Let $g= f^{-1}$. As  each $g^{n_i}\in E(X,f, \Z)$, we can assume without loss of  generality, that $g^{n_i} \to g^q$ uniformly for some $q \in \nat^*$. 
It is well known that $\circ$ is continuous with respect to the uniform topology, so we have that  $f^{n_i}\circ g^{n_i}$ converges uniformly to $f^p\circ g^q$. Therefore $f^p\circ g^q= Id_X$ and analogously $g^q\circ f^p = Id_X$. Then  it is enough to show that $g^q = g^p$.

Let $x\in X$. For each $\epsilon >0$, let us  consider the following sets: 
$$
\begin{array}{ccc}
A_\epsilon &=& \left\{{m \in \nat: f^m(x) \in B(f^p(x); \epsilon)}\right\},\\ 
B_\epsilon &=&  \left\{{m \in \nat: g^m(x) \in B(g^q(x); \epsilon)}\right\}.
\end{array}
$$
Notice that $A_\epsilon\in p$ and $B_\epsilon \in q$.  
We claim  that  $B_\epsilon \in p$ for each $\epsilon$. Suppose otherwise, and let  $\epsilon_0>0$ be such that $B_{\epsilon_0} \notin p$. Pick a strictly increasing sequence $(m_k)_k$ such that  $m_k \in A_{\epsilon_0/k}\setminus  B_{\epsilon_0}$. Thus $f^{m_i}(x) \to f^p(x)$ and this implies by Lemma  \ref{puntual} that $f^{m_i} \to f^p$ pointwise and therefore uniformly  (because both topologies coincide). We can also assume that $(g^{m_i})_i$ is uniformly convergent and, moreover, it has to converge to $(f^p)^{-1}=g^q $ (by the uniform continuity of $\circ$). However, this implies that  there is $i$ such that  $m_i \in B_{\epsilon_0}$, which is a contradiction. Thus $ B_{\epsilon} \in p$ for each $\epsilon>0$. 

Now for each $\epsilon>0$,  consider the following sets:
$$
C_\epsilon=\left\{{m \in \nat: g^m(x) \in B(g^p(x); \epsilon)}\right\}.
$$
Notice that  $C_\epsilon \in p$ for all $\epsilon$.  Pick a sequence $(k_j)_j$ such that $k_j \in B_{1/j}\cap C_{1/j}$ for each $j$. Thus  $g^{k_j}(x) \to g^q(x)$ and  $g^{k_j}(x) \to g^p(x)$ and, by Lemma \ref{puntual}, $g^{k_j} \to g^q$ and  $g^{k_j} \to g^p$. We conclude that $g^p= g^q$.
\end{proof}

\begin{teo}
\label{formulainverso}
Let $X$ be compact metric and $f:X\to X$ be a homeomorphism. Suppose $(X,f)$ is equicontinuous. Then  $(f^p)^{-1} = (f^{-1})^p$ for every $p \in \beta\nat$. 
\end{teo}

\proof 
Let $x\in X$ we need to show that $(f^p\circ (f^{-1})^p)(x)= ((f^{-1})^p\circ f^p)(x)=x$.
If $x$ is periodic, the result follows from Lemma \ref{p-inversos2}. Suppose then that $x$ is not periodic. Since $f$ is injective, then the $\Z$-orbit of $x$ is infinite. Let $Y$ be $\overline{\{f^n(x):\; n\in\Z\}}$ and $g=f\restriction Y$.  By Lemma \ref{inverse-orbita-densa},  the result holds for $(Y,g)$. To finish the proof it suffices to observe that $f^p\restriction Y=g^p$ for every $p\in \beta\nat$. 
\endproof

We have seen in Lemma \ref{p-lim} that the Ellis semigroup is abelian when the system is WAP. But this is not an equivalence  (see Example \ref{abelianNonWAP}). It is natural then to wonder whether  the formula $(f^p)^{-1} = (f^{-1})^p$ is still valid for a distal cascade with a commutative Ellis semigroup.  The answer to this question is affirmative as we show next. We need the following result from  \cite{Downarowicz1998}, we give a proof of it for the sake of completeness.

\begin{teo}\cite[Fact 2.]{Downarowicz1998}\label{abelian 1.}
Let $X$ be a compact metric space and $f: X \to X$ be  a homeomorphism such that the $\Z$-orbit of $x$ is dense on $X$. The following are equivalent.
\begin{enumerate}
    \item[(i)] $E(X,f, \Z)$ is commutative.
    \item[(ii)] $(X,f)$ is WAP.
\end{enumerate}
\end{teo}

\proof By  Lemma \ref{p-lim}, it remains to show that  (i) implies (ii). 
Let $p \in \beta\nat$ and $y\in X$. Let $(y_n)_{n \in \nat}$ be a sequence in $X$  converging to $y$, we want to see that $f^p(y_n) \to f^p(y)$ and also that $({f^{-1}})^p(y_n) \to (f^{-1})^p(y)$. Both claims  are treated analogously and we show  only the first one.
There exists a sequence $(p_k)_k$ in $\beta\nat$ such that  $f^{p_k}(x) = y_k$ for all $k$. As $E(X,f, \Z)$ is compact, we assume that $(p_k)_k$ is a net such that  $f^{p_k} \to f^{q}$ for some  $q \in \beta\nat$. In particular, $f^{p_k}(x) \to f^{q}(x)$ and therefore $f^q(x)=y$. Since the Ellis semigroup is commutative, we have that
$$
f^p(y_{k}) = f^p( f^{p_k}(x)) = f^{p_k}(f^p(x)) \to f^{q}(f^p(x)) = f^p(f^{q}(x))= f^p(y).
$$
This implies that $f^p(y_n) \to f^p(y)$, so the system is WAP.
\endproof

\begin{teo}
Let $(X,f)$ be a dynamical system with $f: X\to X$ a homeomorphism such that $(X,f)$ is distal. If $E(X,f, \Z)$ is abelian, then $(f^p)^{-1}= (f^{-1})^p$  for each $p \in \beta\nat$.
\end{teo}
\proof
Fix $x \in X$ and consider $Y=\overline{\left\{{f^n(x) : n \in \Z}\right\}}$ and let $g = f\restriction Y$. Then $E(Y, g,  \Z)= \{h\restriction Y: h\in E(X,f,\Z)\}$. In fact, one inclusion is obvious as  $f[Y]= Y$. For the other direction, let $p\in \beta\nat$.  Then $g^p=f^p\restriction Y$. Thus   $E(Y, g,  \Z)$ is abelian  and  $(Y,g)$ has a dense orbit, thus  by  Theorem \ref{abelian 1.}, $(Y,g)$ is a WAP system.

Since $(Y,g)$ is distal and WAP, $(Y,g)$ is  equicontinuous  (see \cite{Auslander1988} page 69). Thus, by Theorem \ref{formulainverso} we have $(g^{-1})^p\circ g^p = 1_Y$ for each $p \in \beta\nat$. In particular, as $x\in Y$, 
$x= (g^{-1})^p(g^p(x))= (f^{-1})^p(f^p(x))$.
As $x$ was chosen arbitrarily, $(f^p)^{-1}= (f^{-1})^p$ for any ultrafilter $p$.
\endproof

We do not know whether the previous result holds in general for any distal system. 
Next we give an example of a non WAP distal dynamical system whose Ellis semigroup is abelian. It is the well known system of rotating a disk at different rates (see \cite{Auslander1988} page 65).

\begin{ejem}
\label{abelianNonWAP}
Let $X = \left\{{(x,y) \in \R^2: x^2 +y^2 \leq 2\pi}\right\}$ (a disk on the plane) and let $f:  X\to X$ be given by $f(r,\theta) = (r, \theta+r)$ (in polar coordinates).  Let $S^r$ be the circle of radius $r$  and $R_r:S^r\to S^r$ the rotation with angle $r$, i.e. $R_r(\theta)=f(r,\theta)$.  It is well known that  $R_r$ is  equicontinuous for each $r$. As  $(S^r, R_r)$ is equicontinuous, then it is WAP and thus $E(S^r, R_r)$ is abelian.  Notice also that $S^r$ is invariant under $f$ and moreover  $R_r^p=f^p\restriction S^r$ for each $p\in \beta\nat$ and $0\leq r\leq 2\pi$.  Thus $E(X,f)$ is abelian. 
For the sake of completeness we show that $E(X,f)$ is not WAP.  Let $(n_k)_k$ be an increasing sequence of prime numbers. By the Chinese remainder theorem, the collection $\{(4n_k\nat +2n_k)^*:\;k\in \nat\}$ has the finite intersection property, thus there is  $p \in \bigcap_{k \in \nat}{(4n_k\nat +2n_k)^* }$. We claim that $f^p$ is not continuous. In fact,  the sequence $((\pi +\frac{\pi}{2n_k}, 0))_{k \in \nat}$ converges to $(\pi, 0)$. Each $(\pi +\frac{\pi}{2n_k}, 0)$ has period $4n_k$ and $(\pi, 0)$ has period $2$. 
By Lemma \ref{GRU1} we have $f^p(\pi +\frac{\pi}{2n_k}, 0) = f^{2n_k}(\pi +\frac{\pi}{2n_k}, 0)= (\pi +\frac{\pi}{2n_k}, \pi)$ and $f^p(\pi, 0) = f^2(\pi, 0) = (\pi, 0)$. This shows that $f^p$ is not continuous.
\end{ejem}

When $X$ is countable we have shown that in any distal system $(X,f)$ every point is periodic, thus by Lemma \ref{p-inversos2} $(f^p)^{-1}\in E(X,f,\Z)$  for all $p\in \beta\nat$, that is to say $E(X,f,\Z)$ is a group. This provides a proof of the non trivial part of the following well known theorem. 

\begin{teo} (Ellis)
\label{group}
Let $X$ be a compact metric countable space and  $f:X\to X$ be a homeomorphism. The following are equivalent. 
\begin{itemize}
\item[(i)] $(X,f)$ is distal.

\item[(ii)]  $E(X,f,\Z)$ is a group.
\end{itemize}
\end{teo}


The group $E(X,f,\Z)$ for a distal $f$ can be described as follows. 

\begin{teo}
\label{grupoEllis}
Let $X$ be a compact metric countable space and  $(X,f)$ be a distal dynamical system. Then there is a continuous group monomorphism $\varphi: E(X,f,\Z)\to \prod_{n\in P_f}\mathbb{Z}_n$, where  $\mathbb{Z}_n$ is the group of integers module $n$ with the  discrete topology and  $\prod_{n\in P_f}\mathbb{Z}_n$ has the product topology.
\end{teo}

\proof Let $(r_n)_{n\in P_f}$ with $r_n\in \mathbb{Z}_n$. Suppose $  p,q\in \bigcap_{n\in P_f}(n\nat+r_n)^{\wedge}$. 
Since every point is periodic,  $f^p=f^q$ (see  Lemma \ref{GRU1}).  For each $p\in \beta\nat$ and $n\in P_f$, let $r(p,n)\in \mathbb{Z}_n$ be such that $p\in (n\nat+r(p,n))^{\wedge}$.  Notice that for any $p, q\in \beta\nat$, $f^p=f^q$  if, and only if, $r(p,n)=r(q,n)$ for all $n\in P_f$. Hence the map  $\varphi:E(X,f,\Z)\to \prod_{n\in P_f}\mathbb{Z}_n$ given by 
\[
\varphi(f^p)= (r(p,n))_{n\in P_f} 
\]
is well  defined and clearly injective.

Let us check that  $\varphi$ is continuous. For each   $m\in P_f$,  let $x_m\in X$ be a point with period $m$ and $\pi_m: \prod_{n\in P_f}\mathbb{Z}_n\to \mathbb{Z}_m$  be the $m^{th}$-projection map. Fix $r\in \mathbb{Z}_m$. Then   $(\pi_m\circ \varphi) (f^p) = r$ if, and only if, $p\in (m\nat+r)^{\wedge}$, and this last claim is clearly equivalent to say that   $f^p(x_m)=f^r(x_m)$ (see  Lemma \ref{GRU1}). 
Hence $\pi_m\circ \varphi$ is continuous for all $m$. Thus $\varphi$ is continuous. 

Finally, we check  that $\varphi$ is a homomorphism. 
Let $n\in P_f$ and  $x\in X$ with period $n$. Suppose  $r(p,n)= r$ and $r(q,n)=s$. Then $f^q(y)=f^s(y)$ and $f^p(y)=f^r(y)$ for all $y$ with period $n$. In particular, since $f^q(x)$ has period $n$, $f^p(f^q(x))= f^{r+s}(x)$.  Now notice that if $r+s\equiv t \mod n$, then $f^{r+s}(x)=f^t(x)$. This shows that  
$\varphi(f^p\circ f^q)(n)=\varphi(f^p)(n)+\varphi(f^q)(n)$. 

\endproof

The map $\varphi$ defined in the proof of the previous theorem is not necessarily  onto. For instance, let $f$ be a distal map on $X=\omega +1$  such that $P_f$ is $\{2^n:\; n\geq 1\}$. Let $r_n$ be such that $0\leq r_n< 2^n$ for all $n$ but in addition $r_4=3$ and $r_{16}=9$. Since $(4\nat+3)\cap (16\nat +9)=\emptyset$, there is no $p\in \beta\nat$ such that $\varphi(f^p)=(r_n)_n$. 

When $P_f$ is infinite, $E(X,f,\nat)$ is homeomorphic to $\cantor$ (see Theorem 2.7 of \cite{GRU2016}). 
It is an open problem whether the converse is also true. We can give a partial answer.   Indeed, it follows immediately from Theorem \ref{grupoEllis} that when $(X,f)$ is distal and $P_f$ is finite,  $E(X,f,\nat)$ is finite. 
When $P_f$ and $P_g$ are infinite, the corresponding Ellis semigroups are both homeomorphic to $\cantor$ but they are not necessarily isomorphic. For instance, let $f$ and $g$ be distal maps on $X=\omega +1$ such that  $P_f=\{2^n:\; n\geq 1\}$ and   $P_g=\{3^n:\; n\geq 1\}$
we have that $E(X,f,\nat)$ has  elements of order 2 but 
$E(X,g)$ does not.

\section{Equicontinuous cascades}

We show in this section that for countable spaces distality suffices to get equicontinuity.  We say that a sequence of subsets $(A_n)_n$ of $X$ converges to a set $A$, and we write  $A_n\to A$,  if for every open set $V\supseteq A$, there is $n$ such that $ V\supseteq A_m$ for all $m\geq n$.  A key fact we need is the following result from \cite{GRU2015} which was proved by induction on the cb-rank of $X$. 

\begin{lema}
\label{equico1}
\cite[Corollary 3.5]{GRU2015}
Let $X$ be a compact metric countable space and $f:X\to X$ be continuous such that every point of $X'$ is  periodic. If
$(x_n)_n$ is a sequence of periodic points 
converging to $x$, then $\mathcal O_f(x_n) \to \mathcal O_f(x)$.
\end{lema}

We recall that every countable metric space is zero-dimensional (see e.g. \cite[Corollary 6.2.8]{Engelking}).  Given a distal homeomorphism $f:X\to X$ and  a point $x\in X$ with period $k$, for the next lemmas, we fix pairwise disjoint clopen sets $V_0, V_1,\cdots, V_{k-1}$  such that $\mathcal{O}_f(x)\cap V_r=\{f^r(x)\}$ for every $0\leq r\leq k-1$. Let   $V$ be $V_0\cup\cdots\cup V_{k-1}$.

\begin{lema}
\label{equico2}
Let $X$ be a compact metric countable space and $f:X\to X$  be a distal homeomorphism. 
Let $x_n\to x$ and $0\leq r\leq k-1$. Suppose $(n_i)_i$  and $(m_i)_i$ are sequences of natural numbers  such that $(n_i)_i$ is increasing and  $f^{m_i}(x_{n_i})\in V_r$  for  all $i$. If  $f^{m_i}(x_{n_i}) \to y$, then $y=f^r(x)$. 
\end{lema}

\proof By Theorem \ref{distal}, every point is periodic. Suppose the conclusion does not hold. Then $y\nin \mathcal O_f(x)$. Let $W\su V$ be an open set such that $\mathcal{O}_f(x)\su W$ and $y\not\in\cl W$. Then  $f^{m_i}(x_{n_i})\nin W$ for almost all $i$, which  contradicts Lemma \ref{equico1}.
\endproof

Given  $x\in X$ and $k\in \nat$ with $k\geq 1$, a $k$-{\em segment of the orbit} of $x$ is a set of the form $\{f^m(x), \cdots, f^{m+k-1}(x)\}$ for some $m\in \nat$.

\begin{lema}
\label{equico3} Let $X$ be a compact metric countable space and $f:X\to X$ be  a distal homeomorphism. Let  $x\in X$ with period $k\geq 2$ and $x_n\to x$. There is $n_0$ such that   for all $n\geq n_0$ the following holds:
\begin{itemize}
\item[(i)]  $\{f^m(x_n), f^{m+1}(x_n)\}\not\subseteq V_r$ for all $m$ and  all $0\leq r\leq k-1$.

\item[(ii)]  $S\cap V_r\not = \emptyset$ for every $k$-segment $S$ of the orbit of $x_n$  and all $0\leq r\leq k-1$.

\item[(iii)] $f^m(x_n)\in V_r$, for all $m\in \nat$  with  $m\equiv r \mod k$.

\end{itemize}

\end{lema}

\proof  By Lemma \ref{equico1},  we can assume that  $\mathcal O_f(x_n)\subseteq V$ for all $n$, where $V=V_0\cup\cdots \cup V_{k-1}$.

(i) Let $r\in\{0,\cdots, k-1\}$ and suppose, towards a contradiction,  there are sequences $(n_i)_i$ and  $(m_i)_i$ with $(n_i)_i$ increasing such that $f^{m_i}(x_{n_i}), f^{m_i+1}(x_{n_i})\in  V_r$  for all $i$.  By compactness and   Lemma \ref{equico2}, we can also assume that $f^{m_i}(x_{n_i})\to f^r(x)$.  Thus $f^{m_i+1}(x_{n_i})\to f^{r+1}(x)\in V_j$ for some $j\neq r$,  which contradicts that $f^{m_i+1}(x_{n_i})\in V_r$ for all $i$. 

(ii) By a way of a contradiction, suppose  there are $0\leq r\leq k-1$ and   sequences $(n_i)_i$ ,   $(m_i)_i$ with $(n_i)_i$ increasing such that
\[
S_i=\{f^{m_i}(x_{n_i}), \cdots, f^{m_i+k-1}(x_{n_i})\}
\]
is disjoint from $V_r$ for all $i$. By passing to a subsequence, we assume there is $j\neq r$ such that 
$f^{m_i}(x_{n_i})\in V_j$ for all $i$. By Lemma \ref{equico2}, we can also assume that $f^{m_i}(x_{n_i})\to f^j(x)$.  We consider two cases: (a) Suppose $0\leq j<r$. Then $f^{m_i+r-j}(x_{n_i})\to f^r(x)$ which contradicts that $S_i$ is disjoint from $V_r$. (b) Suppose $r<j\leq k-1$. Then  $0\leq k-j+r\leq k-1$ and $f^{m_i+k-j+r}(x_{n_i})\to f^{k+r}(x)=f^r(x)$ which is a contradiction as before.

(iii) From (i) and (ii) we can assume that for every $k$-segment $S$ of $x_n$  and every $0\leq r\leq k-1$ we have that $S\cap V_r$ has exactly one element.  Let $S_n$ be the first $k$-segment of $x_n$, i.e. 
\[
S_n=\{x_n, f(x_n), \cdots, f^{k-1}(x_n)\}.
\]
We claim that there is $n_0$ such that the conclusion of (iii) holds for $S_n$ and $n\geq n_0$. That is to say,  $f^r(x_n)\in V_r$ for all $0\leq r\leq k-1$ and all $n\geq n_0$. Suppose not. Then we easily find $0\leq r,j\leq k-1$  with $r\neq j$ and an increasing sequence $(n_i)_i$ such that $f^{r}(x_{n_i})\in V_j$ for all $i$. 
We also can assume that  $f^{r}(x_{n_i})$ converges to some $y$, and this contradicts the continuity of $f$. 

Now  consider the second $k$-segment of the orbit of $x_n$, i.e. 
\[
S^2_n=\{f(x_n), f^2(x_n), \cdots, f^{k}(x_n)\}.
\]
 We have just shown that $f^r(x_n)\in V_r$ for all $1\leq r\leq k-1$. Since $S^2_n\cap V_i$ has exactly one element, necessarily $f^k(x_n)\in V_0$.
The rest of the argument follows easily by induction. 
\endproof

\begin{lema}
\label{casiperiodico}
Let $X$ be a compact metric countable space and $f:X\to X$ be  a distal homeomorphism. For all $x\in X$ there is $\delta_x>0$ such that if $d(x,y)<\delta_x$, then 
$f^m(y)\in V_r$, for all $m\in \nat$  with  $m\equiv r \mod k$.
\end{lema}

\proof
Suppose, towards a contradiction, that for all positive integer $l$, there is $y_l\in X$ such that $d(x,y_l)< 1/l$ and there is $m$ such that $f^m(y_l)\nin V_r$ and $m\equiv r \mod k$. Since $y_l\rightarrow x$, this clearly contradicts part (iii) of Lemma \ref{equico3}. 
\endproof

Under the hypothesis of  Lemma \ref{casiperiodico}, it follows immediately that   $\{f^n:\; n\in \nat\}$ is equicontinuous at every point. It is known that equicontinuity is equivalent to uniform almost periodicity (see \cite{Auslander1988} page 36). Now we show that a stronger property holds on countable spaces. 

\begin{teo}
\label{uniforcasiperiodico} Let $X$ be a compact metric countable space and  $f:X\to X$ be a homeomorphism. Then the following are equivalent.

\begin{enumerate}
\item[(i)] Every point is periodic. 
    
\item[(ii)] For all $\varepsilon>0$ there is $l$ such that  $d(x,f^{nl}(x))<\varepsilon$ for all $x\in X$ and all $n\in \nat$ (where $d$ is the metric on $X$).

\item[(iii)] $(X,f)$ is equicontinuous. 

\item[(iv)] There is $p\in \nat^*$ such that $f^p$ is  a homeomorphism.
\end{enumerate} 
\end{teo}

\begin{proof}
(i) $\Rightarrow$ (ii). Let $\varepsilon>0$. For each $x\in X$, let $k_x$ be the period of $x$. Let us fix a collection $\{V^x_r: 0\leq r<k_x\}$ of pairwise disjoint open sets  as we did before.  Clearly, we can assume that each $V^x_k$ has diameter smaller than $\varepsilon$. Let $\delta_x$ be given by the Lemma \ref{casiperiodico}. We can assume that the $\delta_x$-ball $B_x$ around $x$ is a subset of $V^x_0$.  By compactness, there are $x_1, \cdots, x_m$ such that $X=\bigcup_{i=1}^m B_{x_i}$. Let $l$ be the least common multiple of the periods of $x_1, \cdots, x_m$. Let $n\in \nat$ and  $y\in X$. Pick $i$ such that $y\in B_{x_i}$.  From Lemma \ref{casiperiodico} we conclude that $f^{nl}(y)\in V_0^{x_i}$. Hence $d(y,f^{nl}(y))<\varepsilon$.

(ii) $\Rightarrow$ (iii). Let $\varepsilon>0$.  Let $l$ be  such that  $d(x,f^{nl}(x))<\varepsilon/3$ for all $x\in X$ and all $n\in \nat$. Since $f^i$ is uniformly continuous for all $0\leq i<l$, there is  $\delta$ with $0<\delta<\varepsilon/3$ such that if $d(x,y)<\delta$, then $d(f^i(x),f^i(y))<\varepsilon/3$ for $0\leq i<l$. We will show that 
$d(f^m(x),f^m(y))<\varepsilon$ for all $m$ and all $x,y\in X$ with $d(x,y)<\delta$. If fact, let $m\in \nat$ and $0\leq i<l$ such that $m=nl+i$ for some $n\in \nat$. Suppose $d(x,y)<\delta$,  then
\[
d(f^m(x),f^m(y))\leq d(f^{nl}(f^i(x)),f^i(x))+ d(f^i(x),f^i(y))+d(f^{nl}(f^i(y)),f^i(y))<\varepsilon.
\]

(iii) $\Rightarrow$ (iv). It is well known that every equicontinuous system is WAP and  distal, in particular, $E(X,f,\Z)$ is a group of homeomorphism (see \cite{Auslander1988} page 69).  For the sake of completeness, we present a sketch of an argument showing that each $f^p$ is a homeomorphism.  It is quite elementary to verify  that  when $(X,f)$ is equicontinuous, $f^p$ is continuous for any $p\in\beta\nat$. On the other hand, by Theorem \ref{formulainverso}, each $f^p$ is a bijection, and hence a homeomorphism.

(iv) $\Rightarrow$ (i). See Lemma \ref{todosperiodicos}.
\end{proof}

We have now all what is needed to show the result mentioned in the introduction. 

\begin{teo}Let $X$ be a compact metric countable space and  $f:X\to X$ be a homeomorphism. Then the following are equivalent.

\begin{enumerate}

\item[(i)] $(X,f)$ is equicontinuous.

\item[(ii)] $(X,f)$ is a distal.

\item[(iii)] Every point of $X$ is periodic.

\item[(iv)] There is $p\in \nat^*$ such that $f^p$ is a homeomorphism.

\item[(v)] For all $\varepsilon>0$ there is $l$ such that  $d(x,f^{nl}(x))<\varepsilon$ for all $x\in X$ and all $n\in \nat$ (where $d$ is the metric on $X$).

\item[(vi)] $E(X,f,\Z)=\overline{\{f^n:\; n\in\nat\}}$.
\end{enumerate}
\end{teo}

\proof (i) $\Rightarrow$ (ii). It is well known but we skecht the argument for the sake of completeness. We will show that $f^p$ is injective for any $p\in\beta\nat$.  Let $x,y\in X$ with $x\neq y$. Suppose, towards a contradiction, that $f^p(x)=f^p(y)$. Let $0<\varepsilon < d(x,y)$. By equicontinuity, there is  $0<\delta$ such that   if $d(w,z)<\delta$, then $d(f^n(w),f^n(z))<\varepsilon$ for all $n\in \Z$. By the definition of $f^p$,  as $f^p(x)=f^p(y)$, there is $m$ such that $d(f^m(x),f^m(y))<\delta$. Hence $d(x,y)=d(f^{-m}(f^m(x)),f^{-m}(f^m(y)))<\varepsilon$, a contradiction.

(ii) $\Rightarrow$ (iii). See Theorem \ref{distal}.

(iii) $\Rightarrow$ (iv), (iv) $\Rightarrow$ (v) and (v) $\Rightarrow$ (i). See Theorem \ref{uniforcasiperiodico}

(vi) $\Leftrightarrow$ (iii). By Theorem \ref{ZNdistality}.

\endproof

\noindent {\bf Acknowledgment:}
We are  thankful to the referee for his (her)  comments that improved the presentation of the paper. 

\bibliographystyle{plain}

\end{document}